\documentclass{amsart}
\usepackage{amssymb, latexsym}

\DeclareMathOperator{\Tr}{Tr}
\DeclareMathOperator{\Alt}{Alt}
\DeclareMathOperator{\Bil}{Bil}

\DeclareMathOperator{\M}{\mathcal{M}}
\DeclareMathOperator{\N}{\mathcal{N}}

\DeclareMathOperator{\rank}{rank}

\DeclareMathOperator{\rad}{rad}

\DeclareMathOperator{\Symm}{Symm}

\def\matrices{M_{n}(\mathbb{F}_q)}

\theoremstyle{plain}
\newtheorem{theorem}{Theorem}
\newtheorem{corollary}{Corollary}
\newtheorem{lemma}{Lemma}

\theoremstyle{definition}
\newtheorem{definition}{Definition}
\begin{document}

\title[Rank properties of subspaces]
{Rank-related dimension bounds for subspaces   \\ 
of bilinear forms over finite fields}

\author[R. Gow]{Rod Gow}
\address{School of Mathematics and Statistics\\
University College Dublin\\
 Ireland}
\email{rod.gow@ucd.ie}


\begin{abstract} 
Let $q$ be a power of a prime and let $V$ be a vector space of finite dimension $n$
over the field of order $q$. Let $\Bil(V)$ denote the set of all bilinear forms defined on $V\times V$, 
let $\Symm(V)$ denote the subspace of $\Bil(V)$ consisting of symmetric bilinear forms, and $\Alt(V)$ denote the subspace of
alternating bilinear forms. Let $\M$ denote a subspace of any of the spaces $\Bil(V)$, 
$\Symm(V)$, or $\Alt(V)$. In this paper we investigate hypotheses on the rank of the non-zero elements of $\M$ which lead to reasonable
bounds for $\dim \M$. Typically, we look at the case where exactly two or three non-zero ranks occur, one of which is usually $n$.
In the case that $\M$ achieves the maximal dimension predicted by the dimension bound, we try to enumerate the number of forms
of a given rank in $\M$ and describe geometric properties of the radicals of the degenerate elements of $\M$.
\end{abstract}
\maketitle

\section{Introduction}

\noindent Let $q$ be a power of a prime and let $V$ be a vector space of finite dimension $n$
over $\mathbb{F}_q$. A \emph{bilinear form} on $V\times V$ is a function $f:V\times V\to
\mathbb{F}_q$ that is $\mathbb{F}_q$-linear in each of the variables. The \emph{left radical} $\rad_L f$  of $f$ is the subspace 
of all elements $u\in V$ such that $f(u,v)=0$ for all $v\in V$. Similarly, the \emph{right radical} $\rad_R f$  of $f$ is the subspace 
of all elements $w\in V$ such that $f(v,w)=0$ for all $v\in V$. It is known that 
\[
\dim \rad_L f=\dim \rad_R f.
\]
See, for example, Theorem 1.11 of \cite{Art}. We call the integer $\dim V-\dim \rad_L f$ the \emph{rank} of $f$, and we denote it by
$\rank f$. In the case that $f$ is symmetric or alternating, it is clear that the left and right radicals are the same. In this case, 
we call the corresponding subspace the \emph{radical} of $f$ and denote it by $\rad f$. We say that $f$ is degenerate if $\rank f<n$ (equivalently,
if $\rad_L f$ is non-trivial), and non-degenerate if $\rank f=n$.

Let $\Bil(V)$ denote the set of all bilinear forms defined on $V\times V$.  $\Bil(V)$ has the structure of a vector space
of dimension $n^2$ over $\mathbb{F}_q$. 
Let $\Symm(V)$ denote the subspace of $\Bil(V)$ consisting of symmetric bilinear forms, and $\Alt(V)$ denote the subspace of
alternating bilinear forms.

We note that we may identify the vector space $\Bil(V)$ with the vector space $\matrices$, consisting of all $n\times n$ matrices 
with entries in $\mathbb{F}_q$. For if we take a basis  $v_1$, \dots, $v_n$ of $V$, the $n\times n$ matrix 
whose $i,j$-entry is
$f(v_i,v_j)$ completely determines $f$. 

Conversely, let $A$ be an element of $\matrices$. We identify $V$ with $(\mathbb{F}_q)^n$,
considering the elements as column vectors. Then we define an element $f\in \Bil(V)$ by setting
\[
 f(u,v)=u^TAv,
\]
where $u^T$ denotes the transpose of $u$. 

This identification of $\Bil(V)$ with $\matrices$ respects rank, in as much as the rank of an element of $\Bil(V)$ 
equals the usual matrix (row or column) rank of the matrix in $\matrices$ that represents the form. 

We are interested in investigating subspaces $\mathcal{M}$ of the various vector spaces $\Bil(V)$, $\Symm(V)$
and $\Alt(V)$ that satisfy some condition with respect to the rank of each non-zero element of $\mathcal{M}$. From this rank condition,
we hope to find an accurate upper bound for the dimension of
$\mathcal{M}$. 

When the bounds are optimal, and $\dim \M$ reaches the optimal bound, $\M$ often exhibits special properties
of a geometric nature. For example, it may be possible 
to enumerate how many elements of a given rank there are in $\M$. We may also obtain information about $\rad f$ for each element $f$ of $\M$,
or bound the dimension of constant rank subspaces in $\M$. We mention Theorems \ref{dimension_2n}, \ref{dimension_3m}, 
\ref{symmetric_dimension_bound}, \ref{equality_of_dimension}, \ref{spread_of_V}, \ref{three_ranks}, 
\ref{structure_of_2n_dimensional_subspace}, \ref{characteristic_two_bound_for_contained_constant_rank_subspace} as indicators of the type
of theorems proved in this paper.

There is some overlap between certain results proved in this paper and two papers of Kai-Uwe Schmidt, \cite{KUS1}, \cite{KUS2}. 
Most of Theorem \ref{three_ranks} and parts of Theorem \ref{equality_of_dimension} are special cases of Theorem 3.3 of \cite{KUS2}.
The objects of attention in Schmidt's papers are generalizations of subspaces of $\Symm(V)$ in which each non-zero element has rank
at least $r$, where $1\leq r\leq n$. He obtains optimal dimension bounds using 
association scheme techniques but it is interesting to observe that he encounters similar parity
distinctions as we encounter when trying to enumerate the number of forms of a given rank. See, for instance, Case (b) of Theorem
\ref{equality_of_dimension}
and Theorem \ref{structure_of_2n_dimensional_subspace_n_even}, where our findings are inequalities, not necessarily equalities
as in other cases.

\section{Properties of bilinear forms}

\noindent We begin by making a definition that relates subspaces of $V$ with subspaces of bilinear forms.

\begin{definition} Let $U$ be a subspace of $V$ and let $\M$ be a subspace of $\Bil(V)$. We set $\M^L_U$ to be
 the subspace of all elements $f$ of $\M$ that contain $U$ in $\rad_L f$. We likewise set $\M^R_U$ to be
 the subspace of all elements $f$ of $\M$ that contain $U$ in $\rad_R f$. 
 
 In the case that $U$ is one-dimensional and spanned, say, by $u$, we denote the two subspaces by $\M^L_u$ and $\M^R_u$, respectively.
 
 If $\M$ consists of alternating or symmetric forms, we denote the subspaces by $\M_U$ or $\M_u$.
\end{definition}

\begin{lemma} \label{subspace_in_radical}
 Let $U$ be a subspace of $V$ of dimension $r$, where $0<r<n$. 
 Let $\M$ be a subspace of either alternating or symmetric bilinear forms defined on $V\times V$ and suppose that
 $\dim \M\geq rn+\epsilon$, where $\epsilon\geq 0$. 
 
 Then if $\M\leq \Alt(V)$, we have
 \[
  \dim \M_U\geq \frac{r(r+1)}{2}+\epsilon.
 \]
Similarly, if $\M\leq \Symm(V)$, we have
 \[
  \dim \M_U\geq \frac{r(r-1)}{2}+\epsilon.
 \]
\end{lemma}

\begin{proof}
 We may identify $\Alt(V)_U$ as the space of all alternating bilinear forms defined on a vector space of dimension $n-r$. Thus
 \[
  \dim \Alt(V)_U=\frac{(n-r)(n-r-1)}{2}.
 \]
Likewise, 
 \[
  \dim \Symm(V)_U=\frac{(n-r)(n-r+1)}{2}.
 \]
 Suppose now that $\M\leq \Alt(V)$. Then $\M_U=\M\cap  \Alt(V)_U$ and the lower bound for the dimension follows from the inequality
 \[
  \dim \M\cap  \Alt(V)_U\geq \dim \M+\dim \Alt(V)_U-\dim \Alt(V).
 \]
An analogous proof applies when $\M$ consists of symmetric forms.
\end{proof}

We note that the result above holds when $V$ is a vector space defined over an arbitrary field, not just a finite one.

We mention a simple corollary.

\begin{corollary} \label{two_dimensional_subspaces}
 
 Let $\M$ be a subspace of $\Symm(V)$ of dimension $2n$ in which each non-zero element has rank at
 least $n-2$. Let $U$ be a two-dimensional subspace of $V$. Then there exists an element of rank $n-2$ in $\M$ whose
 radical is $U$.
\end{corollary}

\begin{proof}
 
 We take $r=2$ and $\epsilon=0$  in Lemma \ref{subspace_in_radical}. Then we find that $\dim \M_U\geq 1$. Now any non-zero
 element in $\M_U$ has rank at most $n-2$, and hence by our hypothesis on $\M$, it has rank exactly $n-2$. 
\end{proof}

Here again, the corollary is true when $V$ is defined over an arbitrary field $K$, although subspaces $\M$ with the desired
dimension and rank properties do not necessarily exist for all $K$. They exist if $K$ is finite, for example, as we shall
see in Section 5.

It is possible to obtain similar estimates for $\dim \M^L_U$ and $\dim \M^R_U$ when the forms in $\M$ are not necessarily alternating or
symmetric. We will content ourselves with restricting to the case that $U$ is one-dimensional. 

Rather than following the approach
developed in Lemma \ref{subspace_in_radical}, 
we will make use of the following simple principle of linear algebra. Let $V^*$ denote the dual space of $V$,
let $\M$ be a subspace of $\Bil(V)$ and let $u$ be any non-zero element of $V$. We define a linear transformation $\epsilon_u:\M\to V^*$
by setting
\[
 \epsilon_u(f)(v)=f(u,v)
\]
for all $f\in\M$ and all $v\in V$. 
It is clear that $\M^L_u$ is the kernel of $\epsilon_u$. We may similarly define a linear transformation $\eta_u:\M\to V^*$
by setting
\[
 \eta_u(f)(v)=f(v,u)
\]
for all $f\in\M$ and all $v\in V$. Here, $\M^R_u$ is the kernel of $\eta_u$.

We note that there is no reason to assume that $\dim \M^L_u=\dim \M^R_u$ but
both subspaces satisfy the following trivial estimate of dimension.

\begin{lemma} \label{estimate_for_dimension_of_kernel}
 
 For each non-zero element $u$ of $V$, we have $\dim \M^L_u\geq \dim \M-n$ and $\dim \M^R_u\geq \dim \M-n$.
\end{lemma}

\begin{proof}
 It will suffice to consider only $\M^L_u$.
 The image of $\M$ under $\epsilon_u$ is a subspace of $V^*$ and hence has dimension at most $n$. Since $\M^L_u$ is the kernel of
 $\epsilon_u$, the result follows from the rank-nullity theorem. 
\end{proof}

To begin our analysis, we describe a numerical invariant of any subspace
$\mathcal{M}$ of $\Bil(V)$ that relates to the distribution of ranks among the elements of $\mathcal{M}$ and also to the dimension
of the subspaces $\M^L_u$ and $\M^R_u$.

\begin{definition} \label{vanishing_elements}

Let $\mathcal{M}$ be a subspace  of $\Bil(V)$.
Let $Z(\M)$ denote the set
of ordered pairs $(u,v)$ in $V\times V$ that satisfy 
\[
 f(u,v)=0
\]
for all $f\in \mathcal{M}$.
  
\end{definition}

This definition makes sense for a subspace of $\Bil(V)$ when $V$ is a vector space over an arbitrary field, not just
a finite one. However, one of the main tools used in this paper is the following counting theorem that only makes sense
for finite fields.

\begin{theorem} \label{common_zeros}
Let $q$ be a power of a prime and let $V$ be a vector space of finite dimension $n$
over $\mathbb{F}_q$. Let $\mathcal{M}$ be a subspace  of $\Bil(V)$, with $\dim \mathcal{M}=d$. 
For each integer $k$ satisfying $0\leq k\leq n$, let $A_k$ denote the number of elements of rank $k$ in $\mathcal{M}$. For
each element $u$ of $V$, let $d(u)=\dim \M^L_u$ and $e(u)=\dim \M^R_u$.
Let $Z(\M)$ be as in Definition \ref{vanishing_elements}. Then we have
\[
 q^{d-n}|Z(\M)|=\Sigma_{u\in V}q^{d(u)}=\Sigma_{u\in V}q^{e(u)}=\Sigma_{0\leq k\leq n}A_kq^{n-k}.
\]
\end{theorem}

\begin{proof}
Let $\Omega$ denote the set of ordered pairs $(f, u)$, where $f\in \mathcal{M}$
and $u\in \rad_L f$. We intend to calculate
$|\Omega|$ by double counting. We begin by fixing $f\in \mathcal{M}$, with  $\rank f=k$. 
 We have $\dim \rad_L f=n-k$, since $f$ has rank $k$. Thus there are $q^{n-k}$ elements of the form $(f,u)$ in $\Omega$.
 Summing over the various ranks, we obtain
 \[
  |\Omega|=\Sigma_{0\leq k\leq n}A_kq^{n-k}.
 \]

 Next, we fix $u$ in $V$ and note that there are $q^{d(u)}$ elements $f$ with $(f,u)\in \Omega$. Thus, summing over all elements of $V$, we obtain
 \[
  |\Omega|=\Sigma_{u\in V}q^{d(u)}=\Sigma_{0\leq k\leq n}A_kq^{n-k}.
 \]

Finally, we count how many pairs $(u,v)$ are in $Z(\M)$. We fix $u$ and observe that those $v$ in $Z(\M)$ are the vectors annihilated
by the subspace $\epsilon_u(\M)$ in $V^*$. The dimension of the subspace of vectors annihilated by $\epsilon_u(\M)$ is 
\[
\dim V^*-\dim \epsilon_u(\M).
\]
We also know that
\[
 \dim \epsilon_u(\M)=\dim \M-\dim \M^L_u=d-d(u).
\]
Thus the dimension of the subspace of vectors annihilated by $\epsilon_u(\M)$ is $n-d+d(u)$. 
When we sum over all elements of $V$, we obtain
\[
 |Z(\M)|=\Sigma_{u\in V}q^{n-d+d(u)}.
\]

We compare these formulae and thereby derive the equalities
\[
 q^{d-n}|Z(\M)|=\Sigma_{u\in V}q^{d(u)}=\Sigma_{0\leq k\leq n}A_kq^{n-k},
\]
as required. 

To obtain a formula involving the powers $q^{e(u)}$, we count ordered pairs $(f, u)$, where $f\in \mathcal{M}$
and $u\in \rad_R f$ and proceed in an identical manner to that used above.
\end{proof}

It is clear that given any vectors $u$ and $v$ in $V$, both $(u,0)$ and $(0,v)$ are in $Z(\M)$. These are the trivial elements
of $Z(\M)$. Subspaces $\M$ for which $Z(\M)$ consists of these trivial elements have special interest in the theory of bilinear forms.

We can characterize such subspaces as follows.

\begin{corollary} \label{condition_for_only_trivial_elements}

Let $\mathcal{M}$ be a subspace  of $\Bil(V)$, with $\dim \mathcal{M}=d$. Then $Z(\M)$ consists only of trivial elements
if and only if $\dim \M^L_u=\dim \M^R_u=d-n$ for all non-zero elements $u$ of $V$.
 
\end{corollary}

\begin{proof}
 
 We know that $Z(\M)$ contains no non-trivial elements if and only if $\epsilon_u(\M)=V^*$ for all non-zero $u$ in $V$. Since
 $\dim \epsilon_u(\M)=d-\dim \M^L_u$, we see that this occurs if and only if $\dim \M^L_u=d-n$. Similarly, $Z(\M)$ contains no 
 non-trivial elements if and only if
 $\eta_u(\M)=V^*$ for all non-zero $u$ in $V$, and this occurs if and only if $\dim \M^R_u=d-n$
\end{proof}

In the case of the finite field $\mathbb{F}_q$, we have the general bound $|Z(\M)|\geq 2q^n-1$ and equality occurs
in the bound if and only if $\dim \M^L_u=\dim \M^R_u=d-n$ for all non-zero elements $u$ of $V$, by Corollary \ref{condition_for_only_trivial_elements}.

We can now obtain a dimension bound for a \emph{constant rank} $r$ subspace of bilinear forms, that is, a 
non-zero subspace where the rank of each non-zero element is a constant, $r$, where $1\leq r\leq n$.

\begin{theorem} \label{dimension_bound}
Let $\mathcal{M}$ be a constant rank $r$ subspace of $\Bil(V)$. Then we have
\[
 \dim \mathcal{M}\leq 2n-r.
\]

\end{theorem}

\begin{proof}
We set $d=\dim \mathcal{M}$. Following the notation of Theorem \ref{common_zeros}, we have $A_0=1$, $A_r=q^d-1$, and all other values $A_k$ are 0. Thus,
\[
 |Z(\M)|=q^{2n-d}+(q^d-1)q^{2n-d-r}=q^{2n-d}+q^{2n-r}-q^{2n-d-r}.
\]
To prove the dimension bound, it suffices to show that we can not have $d=2n-r+1$. 

Suppose then that $d=2n-r+1$. We obtain
\[
 |Z(\M)|=q^{r-1}+q^{2n-r}-q^{-1}.
\]
This is clearly impossible, as $|Z(M)|$ is a positive integer, but $q^{r-1}+q^{2n-r}-q^{-1}$ is not an integer. It follows that
$d\leq 2n-r$. 

\end{proof}

The dimension bound given in Theorem \ref{dimension_bound} is generally not optimal, although it has the advantage that it applies
for all values of $q$. In a previous paper we obtained the following better upper bound for a constant rank 
$r$ subspace of $\Bil(V)$ provided $q\geq r+1$.

\begin{theorem} \label{constant_rank_dimension_bound}
 
Let $\mathcal{M}$ be a constant rank $r$ subspace of $\Bil(V)$. Then we have
\[
 \dim \mathcal{M}\leq n
\]
provided that $q\geq r+1$. 
\end{theorem}

In proving this theorem, we use the identification of $\Bil(V)$ with $\matrices$. 

Our dimension bound for a constant rank subspace of $\Bil(V)$ enables us to obtain the following generalization.

\begin{theorem} \label{two_rank_subspace}
 
Let $\M$ be a subspace of $\Bil(V)$ and suppose that the rank of each non-zero element of $\M$ is either $r$ or $n$, where $r<n$. Then we have
$\dim \M\leq 3n-r$. Moreover, if $q\geq r+1$, we have the improved bound $\dim \M\leq 2n$. 
\end{theorem}

\begin{proof}
We use the notation introduced in Lemma \ref{estimate_for_dimension_of_kernel}. Let $u$ be a non-zero element of $V$. The non-zero
elements of $\M^L_u$ are degenerate, and hence $\M^L_u$ is a constant rank $r$ subspace. 

Theorem \ref{dimension_bound} now implies that $\dim \M^L_u\leq 2n-r$. Thus $\dim \M\leq \dim M^L_u+n\leq 3n-r$ for all values of $q$.
If we assume that $q\geq r+1$, we have $\dim \M^L_u\leq n$ by Theorem \ref{constant_rank_dimension_bound}. Thus $\dim \M\leq n+n=2n$, as claimed.
\end{proof}

We would like to consider in more detail what occurs in Theorem \ref{two_rank_subspace} when $\dim \M=2n$, and $q\geq r+1$. We require
a preliminary lemma, which we prove with the aid of Zsigmondy's theorem, although we suspect a more elementary proof can be devised.

\begin{lemma} \label{zsigmondy}
 
Let $a\geq 2$ be an integer and let $r$, $s$ and $t$ be positive integers. Suppose that $a^r-1$ divides $(a^s-1)^t$. Then either $r$ divides $s$
or $r=2$, $a+1$ is a power of $2$,  $s$ is odd and $t$ is sufficiently large.  
\end{lemma}

\begin{proof}
 We may clearly assume that $r>1$. By Zsigmondy's theorem, with certain exceptions which we shall consider later, there exists a prime $p$, say, that divides $a^r-1$ but does not divide $a^u-1$ for any integer $u$ satisfying $1\leq u<r$. 

Suppose that we are in a non-exceptional case of Zsigmondy's theorem, and let $p$ be a prime divisor of $a^r-1$ with the property described above. Then $a$ has order $r$ modulo $p$. Now as $p$ divides $a^r-1$, it also divides $(a^s-1)^t$ by hypothesis, and hence divides $a^s-1$ also.
Thus $a^s\equiv 1\bmod p$. This implies that $r$ divides $s$, since $r$ is the order of $a$ modulo $p$.

We turn to the exceptional cases of Zsigmondy's theorem. When $a=2$ and $n=6$, no such $p$ exists. Now $2^6-1=63$ and we are thus assuming that
$63$ divides $(2^s-1)^t$. This implies that $3$ and $7$ both divide $2^s-1$. The order of $2$ modulo 3 is $2$ and the order of $2$ modulo $7$ is $3$. 
Then $2$ divides $s$, since $2^s\equiv 1\bmod 3$, and $3$ divides $s$, since $2^s\equiv 1\bmod 7$. Hence $6$ divides $s$, as required.

The other exception to Zsigmondy's theorem occurs when $r=2$ and $a+1$ is a power of $2$, say $a+1=2^v$, for some integer $v\geq 2$. 
Then $a^2-1=2^v(a-1)$ and hence $a^2-1$ will divide $(a^s-1)^t$ whenever $s$ is odd and $t\geq v+1$. This explains the exceptional case in the lemma.
\end{proof}

\begin{theorem} \label{dimension_2n}
Let $\M$ be a subspace of $\Bil(V)$ of dimension $2n$ and suppose that the rank of each non-zero element of $\M$ is either $r$ or $n$, where $0<r<n$.
Suppose also that $q\geq r+1$. Then $n-r$ divides $n$ and the number of elements of rank $r$ in $\M$ is
\[
 \frac{(q^n-1)^2}{(q^{n-r}-1)}.
\]

\end{theorem}

\begin{proof}
 
The proof of Theorem \ref{two_rank_subspace} shows that $\dim \M^L_u=d(u)=n$ for each non-zero element $u$ of $V$. Theorem \ref{common_zeros}
now implies that
\[
 (q^n-1)q^n+q^{2n}=q^n+A_rq^{n-r}+A_n.
\]
We also have $A_n=q^{2n}-1-A_r$. On substitution, we derive that
\[
 q^{2n}-2q^n+1=A_r(q^{n-r}-1)
\]
and this yields the stated formula for $A_r$.

Necessarily, then, $q^{n-r}-1$ divides $(q^n-1)^2$ and hence $n-r$ divides $n$, by Lemma \ref{zsigmondy} (the exceptional case does not apply
since $t=2$ here). 

\end{proof}

We show finally that the conclusion of Theorem \ref{dimension_2n} is optimal. 

\begin{theorem} \label{realizing_dimension_2n}
Let $r$ be an integer that satisfies $0<r<n$ and suppose that $n-r$ divides $n$. Then $\Bil(V)$ contains a subspace of dimension $2n$
in which the rank of each non-zero element is  either $r$ or $n$.

\end{theorem}

\begin{proof}
 
 We set $m=n-r$ and define the integer $s$ by $n=ms$. Consider the space $M_{s}(\mathbb{F}_{q^m})$ of $s\times s$ matrices over 
 $\mathbb{F}_{q^m}$. It is well known that this space contains a subspace of dimension $2s$ over $\mathbb{F}_{q^m}$ in which each non-zero
 element has rank $s-1$ or $s$. See, for example, Theorem 6.3 of \cite{D}.
 
 Thus, as we explained earlier, if $W$ denotes a vector space of dimension $s$ over $\mathbb{F}_{q^m}$, we may
 construct a subspace, $\N$, say, of $\Bil(W)$ in which each non-zero
 element has rank $s-1$ or $s$ and $\N$ has dimension $2s$ over $\mathbb{F}_{q^m}$. 
 Let $V$ denote $W$ considered as a vector space of dimension $ms=n$ over $\mathbb{F}_{q}$. 
 
 Given $f$ in $\N$, we define $F$ in $\Bil(V)$ by setting
 \[
  F(u,v)=\Tr(f(u,v))
 \]
for $u$ and $v$ in $V$, where $\Tr$ denotes the  trace form from $\mathbb{F}_{q^m}$ to $\mathbb{F}_{q}$. It is straightforward to
see that if $f$ is non-zero, $F$ has rank $m(s-1)=r$ or $ms=n$. The set of all elements $F$, as $f$ runs over $\N$, then is a subspace of
$\Bil(V)$ of dimension $2ms=2n$ in which each non-zero element has rank $r$ or $n$, as required.
\end{proof}

\section{Alternating bilinear forms}
\noindent Let $\M$ be a subspace of $\Alt(V)$. Then we have $f(u,u)=0$ for all $u\in V$ and all $f\in \M$. Thus, $\epsilon_u(\M)$ annihilates
$u$ and hence cannot be all of $V^*$. Consequently, we have $\dim \epsilon_u(\M)\leq n-1$. This implies the following sharpening
of Lemma \ref{estimate_for_dimension_of_kernel} in the alternating case.

\begin{lemma} \label{improved_estimate_for_dimension_of_kernel}
 Let $\M$ be a subspace of $\Alt(V)$. 
 For each non-zero element $u$ of $V$, $\dim \M_u\geq \dim \M-n+1$.
\end{lemma}

\begin{corollary} \label{rank_n-1}
Let $\mathcal{M}$ be a subspace of  $\Alt(V)$. Suppose that $n=\dim V$ is odd and $\rank f=n-1$ for each non-zero
element $f\in \mathcal{M}$. Then we have
\[
 \dim \mathcal{M}\leq n.
\]
Furthermore, if $\dim \mathcal{M}=n$, we have $\dim \M_u=1$ for all non-zero $u$ in $V$. This means that
each one-dimensional subspace of $V$ occurs as the radical of a unique one-dimensional subspace of $\M$. 

\end{corollary}

\begin{proof}
 
Suppose if possible that $\dim \M>n$. It suffices to assume then that $\dim \M=n+1$ and derive a contradiction. 
Following the notation of Theorem \ref{common_zeros}, we have $A_{n-1}=q^{n+1}-1$. Furthermore, Lemma \ref{improved_estimate_for_dimension_of_kernel}
implies that $d(u)\geq 2$ for all non-zero $u$, and $d(0)=n+1$. 

Theorem \ref{common_zeros} gives 
\[
 \Sigma_{u\in V}q^{d(u)}=q^n+A_{n-1}q=q^n+q^{n+2}-q
\]
This is impossible, since $q^2$ divides the left hand side, whereas it does not divide the right hand side. Hence,
$\dim \M\leq n$.

Suppose now that $\dim \M=n$. Theorem \ref{common_zeros} gives 
\[
 \Sigma_{u\in V}q^{d(u)}=q^n+A_{n-1}q=q^n+q^{n+1}-q.
\]
We know that $d(0)=n$ and $d(u)\geq 1$ for each non-zero $u\in V$. Hence
\[
 \Sigma_{u\in V}q^{d(u)}\geq q^n+(q^n-1)q=q^n+q^{n+1}-q.
\]
However, since the sum above actually equals $q^n+q^{n+1}-q$, we must have $d(u)=1$ for each non-zero $u\in V$.
This means that the subspace spanned by $u$ occurs as the radical of a unique one-dimensional subspace of $\M$, as claimed.
\end{proof}

We remark here that $n$-dimensional constant rank $n-1$ subspaces of  $\Alt(V)$ do exist, so that the second
conclusion of the corollary is not vacuous.

Suppose that $n$ is even, and $\mathcal{M}$ is a constant rank $n$ subspace of $\Alt(V)$. Theorem
\ref{dimension_bound} shows that $\dim \mathcal{M}\leq n-1$. It turns out that this bound is far from optimal, since in fact $\dim \mathcal{M}\leq n/2$. 
This improved bound is optimal, and its proof usually exploits the Pfaffian of a skew-symmetric matrix, and the
Chevalley-Warning theorem. 

We would like to show in this paper how this body of ideas can provide an optimal upper bound for $\dim \mathcal{M}$ in the case that
$n$ is even and $\rank f\geq n-2$ for all non-zero elements $f$ of $\mathcal{M}$. We begin by recalling that a subspace
of $\Alt(V)$ can be identified with a subspace of $n\times n$ skew-symmetric matrices with entries in
$\mathbb{F}_q$. In the case that $q$ is a power of 2, skew-symmetric means symmetric with all diagonal entries equal to zero.

\begin{lemma} \label{Ax_Katz} Let $n=2m$ be an even positive integer greater than $2$ and let $\mathcal{M}$ be a subspace of $\Alt(V)$. 
Let $d=\dim \mathcal{M}$ and suppose that we can express $d$ in the form $d=sm+t$, where $s$ and $t$ are integers with $s\geq 1$ and
$0<t\leq m$. Then the number of elements of rank $n$ in $\mathcal{M}$ is divisible by $q^{s}$.
\end{lemma}

\begin{proof} 
We identify $\mathcal{M}$ with a subspace $M$, say, of $n\times n$ skew-symmetric matrices with entries in
$\mathbb{F}_q$. Under this identification, an element of $\mathcal{M}$ has rank $n$ if and only if the corresponding skew-symmetric matrix in $M$
has non-zero determinant.

Let $S$ be any $n\times n$ skew-symmetric matrix with entries in
$\mathbb{F}_q$ and let
$s_{ij}$ be the $(i,j)$--entry of $S$, where $i<j$.
The theory of the Pfaffian, \cite{Art}, Theorem 3.27,
shows that there is a homogeneous polynomial $Pf$ of degree
$m$ in $m(2m-1)$ variables, whose coefficients lie in the prime field, $\mathbb{F}_p$, 
such that
\[
\det S=Pf(s_{12}, \ldots, s_{2m-1,2m})^2.
\]

Now let $\{X_1, \ldots, X_d\}$ be a basis for $M$. We may
then express any element $S$ of $M$ in the form
\[
S=\lambda_1X_1+\cdots +\lambda_dX_d,
\]
where the $\lambda_i\in \mathbb{F}_q$. The properties of the Pfaffian
previously outlined imply that there is a homogeneous
polynomial $Q$ in $\mathbb{F}_q[z_1, \ldots, z_d]$ of degree $m$  such that
\[
\det S=Q(\lambda_1, \ldots, \lambda_d)^2.
\]
By the Ax-Katz theorem, \cite{Kat}, the number of zeros of $Q$ in $(\mathbb{F}_q)^d$  is divisible by $q^e$, where
\[
 e=\left \lceil{\frac{d-m}{m}}\right \rceil=\left \lceil{s-1+\frac{t}{m}}\right \rceil=s.
\]

 Let $A$ be the set
of elements of determinant zero in $M$ and $B$ be the set
of elements of non-zero determinant in $M$. Now the number of zeros of $Q$ in $(\mathbb{F}_q)^d$ equals the number of elements in $A$, and thus
$q^s$ divides $|A|$. Since $|A|+|B|=|M|=q^{sm+t}$, and $q^s$ divides $q^{sm+t}$, it follows that $q^s$ divides $|B|$. This means that
the number of elements of rank $n$ in $\mathcal{M}$ is divisible by $q^{s}$, as required.
\end{proof}

We can now proceed to a dimension bound for certain two-rank subspaces of $\Alt(V)$.

\begin{theorem} \label{dimension_3m} Let $n=2m$ be an even positive integer greater than $2$ and let $\mathcal{M}$ be a subspace of $\Alt(V)$. 
Suppose that $\rank f\geq n-2$ for all non-zero elements of $\mathcal{M}$. Then we have 
$\dim \mathcal{M}\leq 3m$. 
\end{theorem}

\begin{proof}
We set $d=\dim \mathcal{M}$ and suppose if possible that $d>3m$. Then we may as well assume that $d=3m+1$, and we will proceed to derive a contradiction.
Since the only non-zero ranks of elements in $\mathcal{M}$ are $n-2$ and $n$, Theorem \ref{common_zeros} implies that we have
\[
q^{m+1} |Z(\M)|=q^{2m}+q^2A_{n-2}+A_n.
\]
Now as $d=3m+1$, $q^3$ divides $A_n$ by Lemma \ref{Ax_Katz}. Clearly, $q^3$ divides $q^{m+1}$ and $q^{2m}$. It follows that $q^3$ divides
$q^2A_{n-2}$ and hence $q$ divides $A_{n-2}$. This is a contradiction, since $A_{n-2}+A_n=q^d-1$, and hence $q$ cannot divide both
$A_{n-2}$ and $A_n$. We deduce that $d\leq 3m$. 
\end{proof}

We show at the end of Section 5 that this bound is optimal. 

Delsarte and Goethals proved the following result in \cite{DG}. Let $Y$ be a subset of $\Alt(V)$ with the property that
$\rank (f-g)\geq n-2$ whenever $f$ and $g$ are different elements of $Y$. Then, assuming that $n=\dim V$ is even, $|Y|\leq q^{2n-2}$. (This is just
a special case of their more general results.) When $Y$ is a subspace, it implies that $\dim Y\leq 2n-2$. We do not know if our
Theorem \ref{dimension_3m} bound is attainable by their association scheme-based method, but suspect that some arithmetic-algebraic geometry
method is needed to obtain improved results for subspaces.

\section{Symmetric bilinear forms}

\noindent Let $\mathcal{M}$ be a subspace of $\Symm(V)$. When we impose conditions involving the minimum rank of elements of $\mathcal{M}$,
Theorem \ref{common_zeros} provides some information, but it is usually not very effective on its own.
An additional device we employ is to extend a symmetric bilinear form to an hermitian form, as we shall now explain. 

Let $\mathbb{F}_{q^2}$ denote the finite field of order $q^2$. We may construct $\mathbb{F}_{q^2}$ as the extension field of $\mathbb{F}_{q}$
determined by a root $\epsilon$, say, of an irreducible polynomial $x^2+\lambda x+\mu\in \mathbb{F}_{q}[x]$. We can then realize 
$\mathbb{F}_{q^2}$ as elements of the form $\alpha+\beta \epsilon$, where $\alpha$ and $\beta$ are elements of $\mathbb{F}_{q}$ and
$\epsilon^2 +\lambda \epsilon+\mu=0$. We can choose $\lambda=0$ and take $-\mu$ to be any non-square when $q$ is odd. 

Let $\overline{V}=V\otimes \mathbb{F}_{q^2}$, a vector space of dimension $n$ over $\mathbb{F}_{q^2}$. We can represent vectors in
$\overline{V}$ in the form $u+\epsilon v$, where $u$ and $v$ are in $V$. Given $f\in\mathcal{M}$, we define an hermitian form $\overline{f}$
on $\overline{V}\times \overline{V}$ by setting
\[
 \overline{f}(u+\epsilon v,u_1+\epsilon v_1)=f(u,u_1)+\epsilon^{q+1}f(v,v_1)+\epsilon f(v,u_1)+\epsilon^q f(u,v_1)
\]
for all $u$, $u_1$, $v$ and $v_1$ in $V$. It is easy to see that $\overline{f}$ has the same rank as $f$. 
Let $\overline{\mathcal{M}}$ denote the $\mathbb{F}_{q}$-vector subspace of hermitian forms defined in this way. 
$\overline{\mathcal{M}}$ has the same dimension as $\mathcal{M}$, considered as a vector space over $\mathbb{F}_{q}$. 

Let $N(\M)$ denote the set of elements $u+\epsilon v$ in $\overline{V}$ such that
\[
 \overline{f}(u+\epsilon v, u+\epsilon v)=0
\]
for all $\overline{f}\in \overline{\mathcal{M}}$. It is straightforward to see that the condition 
\[
 \overline{f}(u+\epsilon v, u+\epsilon v)=0
\]
is the same as 
\[
 f(u,u)+\lambda f(u,v)+\mu f(v,v)=0.
\]
Thus we can identify $N(\M)$ with the set of all ordered pairs $(u,v)$ in $V\times V$ such that $f(u,u)+\lambda f(u,v)+\mu f(v,v)=0$.

Our next result shows how to calculate $|N(\M)|$ in terms of the ranks of elements in $\mathcal{M}$.

\begin{theorem} \label{hermitian_zeros}
Let $q$ be a power of a prime and let $V$ be a  vector space of dimension $n$ 
over $\mathbb{F}_q$. Let $\mathcal{M}$ be a subspace of $\Symm(V)$, with $\dim \mathcal{M}=d$. 
For each integer $k$ satisfying $0\leq k\leq n$, let $A_k$ denote the number of elements of rank $k$ in $\mathcal{M}$. Let
$x^2+\lambda x+\mu$ be an irreducible polynomial of degree two in $\mathbb{F}_q[x]$ and let $N(\M)$ denote the set
of ordered pairs $(u,v)$ in $V\times V$ that satisfy 
\[
 f(u,u)+\lambda f(u,v)+\mu f(v,v)=0
\]
for all $f\in \mathcal{M}$. Then we have
\[
 q^{d-n}|N(\M)|=\sum_{k=0}^n (-1)^kA_k q^{n-k}.
\]
\end{theorem}

\begin{proof}
As we have explained above, $|N(\M)|$ equals the number of elements $z\in \overline{V}$ that satisfy $\overline{f}(z,z)=0$ for 
all $\overline{f}\in \overline{\mathcal{M}}$.
Since we have $\rank f=\rank \overline{f}$, it is clear that $\overline{M}$ contains exactly $A_k$ elements of rank $k$, for $0\le k\leq n$.
The desired formula for $|N(\M)|$ now follows from Theorem 1 of \cite{Du}. 
\end{proof}

We note that $|N(\M)|\geq 1$, since $(0,0)\in N(\M)$. Another point to observe is that the formula for $|N(\M)|$ is identical with the 
formula for $|Z(\M)|$
when every element of $\mathcal{M}$ has even rank, and so probably adds nothing to our knowledge when such a rank distribution occurs.

When we investigate subspaces of $\Symm(V)$ over fields of characteristic two, we can apply the theory of alternating bilinear forms,
as we shall proceed to demonstrate next. Let $K$ be a field of characteristic two and let $W$ be a vector space of finite dimension $m$ over $K$.
Let $\M$ be a subspace of $\Symm(W)$ and let $\M_{\Alt}$ be the subspace of alternating bilinear forms in $\M$. It is elementary to see
that $\M_{\Alt}$ has codimension at most $n$ in $\M$. 

Let 
\[
 V(\M)=\{ w\in W: f(w,w)=0 \mbox{ for all } f\in \M\}.
\]
Since we are working over a field of characteristic two, $V(\M)$ is a subspace of $V$. 

Let $\M_1$ be a complement for $\M_{\Alt}$ in $\M$. It is straightforward to see that $V(\M)=V(\M_1)$. Furthermore, if $K$ is perfect,
\[
 \dim V(\M_1)=\dim W-\dim \M_1=\dim W-\dim \M+\dim \M_{\Alt}.
\]
See, for example, the remark on p.5 of \cite{Gow3}. 

Suppose now that $w$ is a non-zero element of $V(\M)$. Since $f(w,w)=0$ for all $f$ in $\M$, $(w,w)$ is a non-trivial element of $Z(\M)$.
Therefore, if we know that $Z(\M)$ consists only of trivial elements, we can make the following deduction.

\begin{theorem} \label{size_of_alternating_subspace}

Let $K$ be a perfect field of characteristic two and let $W$ be a finite-dimensional vector space over $K$. Let $\M$ be a subspace of $\Symm(W)$.
Then if $Z(\M)$ consists only of trivial elements, we have
\[
 \dim \M_{\Alt}=\dim \M-\dim W.
\]

\end{theorem}

\begin{corollary} \label{no_alternating_elements}

Let $K$ be a perfect field of characteristic two and let $W$ be a finite-dimensional vector space over $K$. Let $\M$ be a subspace of $\Symm(W)$
in which each non-zero element has maximal rank $\dim W$. 
Then $\M_{\Alt}=0$, that is, $\M$ contains no non-zero alternating elements.

\end{corollary}

\begin{proof}
 
 We claim that $Z(\M)$ consists only of trivial elements. For, as each non-zero element of $\M$ is non-degenerate, $\epsilon_w$ is an injective
 linear transformation from $\M$ into $W^*$ for all non-zero elements $w$ of $W$. Then, since $\dim \M=\dim W$, it follows that
 $\epsilon_w$ is an isomorphism, and this implies that $Z(\M)$ consists
 only of trivial elements. The corollary is thus a consequence of Theorem \ref{size_of_alternating_subspace}.
\end{proof}

We return to the study of finite fields for the rest of this section. 
We shall show how the formulae in Theorems \ref{common_zeros} and \ref{hermitian_zeros} can be combined to obtain a dimension bound for subspaces of  
$\Symm(V)$ where just two non-zero ranks occur, of opposite parity.

\begin{theorem} \label{symmetric_dimension_bound}
 Suppose that $n\geq 2$ and let $\mathcal{M}$ be a subspace of $\Symm(V)$. 
Let $r$ and $\ell$ be integers of opposite parity with $0< r<\ell\leq n$. 
Suppose also that the rank of any non-zero element of $\mathcal{M}$ is $r$ or $\ell$. Then we have 
$\dim \mathcal{M}\leq 2n-r$, except possibly when $q=2$. 
\end{theorem}

\begin{proof}
 Suppose if possible that $d=\dim \mathcal{M}> 2n-r$. Then we may as well assume that $d= 2n-r+1$, and we shall derive a contradiction
to this assumption. Our hypothesis about the parity of $r$ and $\ell$ implies that  $\ell-r$ is odd. 
Theorem \ref{common_zeros}  shows that 
\[
q^{\ell+1-r} |Z(\M)|=q^{\ell}+A_{r}q^{\ell-r}+A_{\ell}.
\]
Similarly, 
\[
 q^{\ell+1-r}|N(\M)|=q^{\ell}+(-1)^{r}A_{r}q^{\ell-r}+(-1)^{\ell}A_{\ell}.
\]
We note that each of these equations implies that $q$ divides $A_{\ell}$. A simple manipulation shows that
\[
 q^{\ell+1-r}(|Z(\M)|+(-1)^{l+1}|N(\M)|)=(1+(-1)^{\ell+1})q^{\ell} +2A_{r}q^{\ell-r}.
\]
Since $q^{\ell+1-r}$ divides the left hand side, and also divides $q^\ell$, we see that $q^{\ell+1-r}$  divides $2A_rq^{\ell-r}$.
We deduce that $q$ divides $2A_{r}$. Now if $q$ is odd, this implies that $q$ divides $A_r$. This is a contradiction. For
 we now know that $q$ divides $A_r+A_\ell$, and this is impossible, since
\[
A_{r}+A_{\ell}=q^{d}-1.
\]
Likewise, if $q$ is a power of 2 with $q>2$, we see that $A_{r}$ is even, and this is again a contradiction.
It follows that $d=\dim \mathcal{M}\leq 2n-r$, except possibly if $q=2$.
\end{proof}

With a little more work, whose details we omit, we can show that when $\M$ is as above, $\dim \M\leq 2n-r+1$ when $q=2$.
We can improve this to $\dim \M\leq 2n-r$ 
when $r$ is odd and  $\ell=n$ is even, as we now prove.

\begin{theorem} \label{revised_symmetric_dimension_bound}
 Suppose that $n\geq 2$ is an even integer and let $\mathcal{M}$ be a subspace of $\Symm(V)$. 
Suppose also that the rank of any non-zero element of $\mathcal{M}$ is $r$ or $n$, where $r$ is odd. Then we have 
$\dim \mathcal{M}\leq 2n-r$. 
\end{theorem}

\begin{proof}
Theorem \ref{symmetric_dimension_bound} shows that the dimension bound is true if $q\neq 2$. 
Suppose then that $q=2$ and $\dim \mathcal{M}= 2n-r+1$. Following the proof of Theorem \ref{symmetric_dimension_bound}, we obtain
\[
 A_n=(|Z(\M)|+|N(\M)|)2^{n-r}-2^n,\quad A_{r}=|Z(\M)|-|N(\M)|.
\]
Then since $A_n+A_{r}=2^{2n-r+1}-1$, we deduce that
\[
 |Z(\M)|(2^{n-r}+1)+|N(\M)|(2^{n-r}-1)=2^{2n-r+1}+2^n-1.
\]
Now we know that $|Z(\M)|\geq 2^{n+1}-1$  and $|N(\M)|\geq 1$. Thus we obtain
\[
 (2^{n-r}+1)(2^{n+1}-1)+2^{n-r}-1=2^{2n-r+1}+2^{n+1}-2\leq 2^{2n-r+1}+2^n-1.
\]
This leads to the inequality $2^{n+1}\leq 2^n+1$, which is clearly absurd. It follows that $\dim \mathcal{M}\leq 2n-r$ even when $q=2$.
\end{proof}

We can exploit the inequalities for $|Z(\M)|$ and $|N(\M)|$, as we did above, to obtain the values of $A_r$ and $A_{n}$ when $r$ is odd, 
$n$ is even, 
and $\dim \mathcal{M}=2n-r$. We also obtain inequalities for these numbers when $r$ is even, $n$ odd and $\dim \mathcal{M}=2n-r$ but have not been
able to establish equalities.

\begin{theorem} \label{equality_of_dimension}
Let $r$ be an integer satisfying $0<r<n$ whose parity is the opposite of that of $n$. 
 Let $\mathcal{M}$ be subspace of $\Symm(V)$ of dimension $2n-r$ in which the rank of any non-zero element is $r$ or $n$. Then we have
 \[
  A_r\geq q^n-1, \quad  A_n\leq q^{2n-r}-q^n.
 \]
 The following two cases arise. 
 
 \noindent Case (a): $r$ is odd and $n$ even. Then we have
 \[
A_r=q^n-1, \quad  A_n=q^{2n-r}-q^n, \quad |Z(\M)|=2q^n-1,\quad |N(\M)|=1
\]
and hence for each non-zero element $u$ of $V$, 
$\dim \M_u=n-r$. If $q$ is a power of $2$, $\dim \M_{\Alt}=n-r$. 

\noindent Case (b): $r$ is even and $n$ odd. Then we have $|N(\M)|\geq 2q^r-1$. Furthermore, $A_r=q^n-1$ if and only if
$|N(\M)|=2q^r-1$ if and only if $|Z(\M)|=2q^n-1$.

\end{theorem}

\begin{proof}

Since $A_r+A_n=q^{2n-r}-1$, Theorem \ref{common_zeros} shows that
\[
 \Sigma_{u\in V}(q^{d(u)}-1)=A_r(q^{n-r}-1),
\]
where the sum extends over all non-zero vectors $u$ of $V$.

Lemma \ref{estimate_for_dimension_of_kernel} implies that $d(u)\geq n-r$ and we deduce that the sum on the left above is at least
$(q^n-1)(q^{n-r}-1)$. This yields the inequality $A_r\geq q^n-1$, as required. The inequality for
$A_n$ follows from the inequality for $A_r$.

Next, we consider case (a), where $r$ is odd and $n$ even. 
When we take $\dim \mathcal{M}=2n-r$, our previous formulae give us
\[
 2A_n=q^{n-r}(|Z(\M)|+|N(\M)|)-2q^n.
\]
Thus, since $|Z(\M)|+|N(\M)|\geq 2q^n$,  we deduce that $A_n\geq q^{2n-r}-q^n$. Since we already know that $A_n\leq q^{2n-r}-q^n$,
we obtain that $A_n=q^{2n-r}-q^n$. The  equalities for $A_r$, $A_n$, $|Z(\M)|$ and $|N(\M)$ follow from this equation. Furthermore, since
$Z(\M)$ consists only of trivial elements, $\dim \M_{\Alt}=n-r$ when $q$ is a power of 2, by Theorem \ref{size_of_alternating_subspace}.

Finally, we consider case (b), where $r$ is even and $n$ odd. Our formulae for $|Z(\M)|$ and $|N(\M)|$ and the inequality for $A_n$ yield
\[
 2A_n=q^{n-r}(|Z(\M)|-|N(\M)|)\leq 2(q^{2n-r}-q^n).
\]
This leads to the inequality $|N(\M)|\geq |Z(\M)|-2q^n+2q^r$. Since the trivial estimate $|Z(\M)|\geq 2q^n-1$ holds,
we see that $|N(\M)|\geq 2q^r-1$. Equality holds here if and only if $|Z(\M)|=2q^n-1$ if and only if $A_r=q^n-1$.

\end{proof}

The value for $|N(\M)|$ in case (a) and inequality for $ |N(\M)|$ in case (b) show that there
is an essential difference between the two cases.

We would like to describe an interesting decomposition of subspaces of $\M$ and $V$ that occurs in relation to Case (a) of 
Theorem \ref{equality_of_dimension}. Suppose that we take  $r=n/2$, where $r$ is odd, in the theorem.
The subspaces $\M_u$ are then $r$-dimensional constant rank $r$ subspaces in $\Symm(V)$, and this condition is very special, as shall now show.

\begin{theorem} \label{spread_of_V}
 Suppose that $r$ is an odd positive integer and let $n=2r$. 
Let $\mathcal{M}$ be subspace of $\Symm(V)$ of dimension $3r$ in which
the rank of any non-zero element is $r$ or $n=2r$. Suppose also that $q$ is odd and at least $r+1$. 

Then the non-zero elements of each subspace $\M_u$ have the same $r$-dimensional radical. Furthermore, the subspaces $\M_u$ are either identical
or intersect trivially. 

The radicals of the non-zero elements of the subspaces $\M_u$ form a spread of $q^r+1$ $r$-dimensional subspaces that cover $V$. Thus 
there are $q^r+1$ different subspaces $\M_u$, and any $r$-dimensional constant rank $r$ subspace of $\M$ has this form.

Let $\M_w$ be an $r$-dimensional subspace different from $\M_u$. 
Then the $n$-dimensional subspace $\M_u\oplus \M_w$ contains no elements of rank $r$ other than the elements in $\M_u$ 
and $\M_w$. Thus 
\[
(\M_u\oplus \M_w)\cap \M_z=0,\quad \M=\M_u\oplus \M_w\oplus \M_z
\] 
for any $r$-dimensional constant rank $r$ subspace $\M_z$ different from $\M_u$ and $\M_w$.
\end{theorem}

\begin{proof}

Let $u$ be a non-zero element of $V$. Theorem \ref{equality_of_dimension} shows that $\dim \M_u=r$. Thus, since the non-zero elements of $\M_u$
have rank $r$, $\M_u$ is an $r$-dimensional constant rank $r$ subspace of $\M$. 
 Since we are assuming that $q$ is odd and is at least $r+1$,  
Theorem 2 of \cite{Gow2} shows that all the non-zero elements in $\M_u$ have the same radical, $U$, say, where $U$ is an $r$-dimensional subspace of 
$V$. 

Let $\N$ be any $r$-dimensional constant rank $r$ subspace of $\M$. As we have noted above, all non-zero elements of $\N$ have the same radical.
Let $v$ be any non-zero element of this radical. Then each element of $\N$ is in $\M_v$, and the equality of dimensions implies that
$\N=\M_v$. 

Let $w$ be a different non-zero element of $V$. We claim that $\M_u=\M_w$ or $\M_u\cap \M_w=0$. 
For suppose that
$f$ is a non-zero element in $\M_u\cap \M_w$. Then since all the elements of $\M_w$ have the same radical, again by Theorem 2 of \cite{Gow2}, 
and $f$ is an element of $\M_u$ with radical $U$, it follows that all elements of $\M_w$ also have radical equal to $U$. Now if we 
consider the subspace $\M_u+\M_w$ of $\M$, we see that $U$ is contained in the radical of each element of $\M_u+\M_w$. This implies, by the rank 
properties of the elements of $\M$, that $\M_u+\M_w$ is a constant rank $r$ subspace. Since the maximum dimension of a constant
rank $r$ subspace is $r$, and $\dim \M_u+\M_w\geq r$, we deduce that we must have $\M_u=\M_w$. 

Suppose, on the other hand, that $\M_u\cap \M_w=0$. Let $U$, as above, and $U_1$ be the common radicals of the non-zero
elements of $\M_u$, $\M_w$, respectively. We claim that $U\cap U_1=0$. For suppose, if possible, that $x$ is a non-zero element of 
$U\cap U_1=0$. Then $x$ is in the radical of all elements of $\M_u$ and of $\M_w$, and hence is in the radical of
all elements of $\M_u+\M_w$. This implies that all non-zero elements of $\M_u+\M_w$ have rank $r$, and this is impossible, since
$\M_u+\M_w$ has dimension $2r$. It follows that $U\cap U_1=0$, as claimed. We see then that the different radicals form a spread 
of $r$-dimensional subspaces covering $V$, since each non-zero element of $V$ is in the radical of some element of rank $r$ in $\M$.

Continuing with the notation above, we observe that since $U$ and $U_1$ both have dimension $r$ and intersect trivially, we have $V=U\oplus U_1$. 
Let $f$ and $f_1$ be non-zero elements of $\M_u$
and $\M_w$, respectively. We claim that $f+f_1$ has rank $n$. For, since $f_1$ is identically zero on $U_1\times U_1$,
and has rank $r$ on $V\times V$, it has rank $r$ on $U\times U$. Likewise, $f$ has rank $r$ on $U_1 \times U_1$ and is identically zero on
$U\times U$. Thus since $V=U\oplus U_1$ and this decomposition is orthogonal with respect to $f$ and $f_1$, we deduce that $f+f_1$ has rank $2r=n$.
This implies that
the only elements of rank $r$ in $\M_u\oplus \M_w$ are those in $\M_u$ or in $\M_w$.

Finally, let $\M_z$ be an $r$-dimensional constant rank $r$ subspace different from both $\M_u$ and $\M_w$. As we have seen that
the only elements of rank $r$ in $\M_u\oplus \M_w$ are those in $\M_u$ or in $\M_w$, and $M_z$ contains no elements of rank $r$ in common
with $\M_u$ and $\M_w$, we deduce that $(\M_u\oplus \M_w)\cap \M_z=0$. A count of dimensions then reveals that 
$\M_u\oplus \M_w\oplus \M_z=\M$.

\end{proof}

Theorem \ref{spread_of_V} also holds when $q$ is a power of 2 greater than $r$ (and more generally for any infinite field
 of characteristic 2), as we have shown in \cite{Gow3}, Theorem 6. We suspect that the theorem also holds more generally.
 As an indicator of this possibility, we present a proof of the following rather special result.
 
 \begin{theorem} \label{n=6_case}
 
 Let $K$ be an arbitrary field, subject to the condition that $|K|\geq 4$, and let $W$ be a six-dimensional vector space over $K$. Let
 $\M$ be a subspace of $\Symm(W)$ in which each non-zero element has rank $3$ or $6$. Then $\dim \M\leq 9$.
 
 Furthermore, suppose that $\dim \M=9$. Then for each non-zero element $u$ of $W$, we have $\dim \M_u=3$ and $\M_u$ is a constant rank
 $3$ subspace, whose non-zero elements all have the same radical. As $u$ ranges over $W$, the set of different radicals so obtained is a spread
 of three-dimensional subspaces covering $W$.
 
 Any constant rank $3$ subspace of $\M$ is contained in a subspace of the form $\M_u$. Thus the $\M_u$ are the maximal constant rank
 $3$-subspaces of $\M$. 
 
 Let $\M_w$ be a subspace of $\M$ different from $\M_u$. 
Then the $6$-dimensional subspace $\M_u\oplus \M_w$ contains no elements of rank $3$ other than the elements in $\M_u$ 
and $\M_w$. Thus 
\[
(\M_u\oplus \M_w)\cap \M_z=0,\quad \M=\M_u\oplus \M_w\oplus \M_z
\] 
for any $3$-dimensional constant rank $3$ subspace $\M_z$ different from $\M_u$ and $\M_w$.

 \end{theorem}
 
 \begin{proof}
Lemma \ref{estimate_for_dimension_of_kernel} implies that $\dim \M_u\geq \dim \M-6$. It is also
clear from the rank properties of elements of $\M$ that $\M_u$ is a constant rank 3 subspace of $\M$. We deduce from Corollary 9 of \cite{Du} 
that $\dim \M_u\leq 3$ and hence $\dim \M\leq 9$.

Suppose from now on that $\dim \M=9$. The argument above implies that $\dim \M_u=3$ and Corollary 9  of \cite{Du} yields
that all non-zero elements of $\M_u$ have the same radical. 

Virtually all the remainder of the proof follows that of Theorem \ref{spread_of_V}. The only part that requires further attention
is as follows. Let $\N$ be a non-zero constant rank 3 subspace of $\M$. Suppose first that $\dim \N=1$ and let $f$ span $\N$. Let
$w$ be a non-zero element of $\rad f$. Then $f\in \M_w$ and hence $\N$ is  contained in $\M_w$.

Suppose next that $\dim \N=2$. It follows from Lemma 4 of \cite{Du} that all non-zero elements of $\N$ have the same radical. If $v$ is 
a non-zero element of this common radical, $\N$ is contained in $\M_v$. This completes the proof.
  
 \end{proof}

We now provide examples which show that, in some cases, the bound described in Theorem \ref{symmetric_dimension_bound} is optimal.

\begin{theorem} \label{bounds_attained} Let $r$ be an integer that satisfies $0<r<n$ and suppose that $n-r$ divides $n$. 
Then $\Symm(V)$ contains a subspace of dimension $2n-r$
in which the rank of each non-zero element is  either $r$ or $n$.

\end{theorem}

\begin{proof}

 We set $m=n-r$ and define the integer $s$ by $n=ms$. 
 Let $L$ denote the field of order $q^{m(s+1)}$, which we consider to be
 a vector space of dimension $s+1$ over $\mathbb{F}_{q^m}$, and let $\Tr:L\to \mathbb{F}_{q^m}$ denote the usual trace form. 
 For each element $z$ of $L$,
 we define an element $f_z$, say, of $\Symm(L)$ by setting
\[
 f_z(x,y)=\Tr(z(xy))
\]
for all $x$ and $y$ in $L$. Each form $f_z$ is non-degenerate if $z\neq 0$, since the trace form is non-zero.

Let $W$ be any $\mathbb{F}_{q^m}$-subspace of dimension $s$ in $L$, and let $\mathcal{N}$ denote the restriction of the subspace of bilinear forms
$f_z$ to $W\times W$. Assume now that $z\neq 0$. Then since $f_z$ is non-degenerate, its restriction to $W\times W$ has rank $s-1$ or $s$. 
Therefore, $\mathcal{N}$ is an $(s+1)$-dimensional subspace of  $\Symm(W)$, in which each non-zero
element has rank $s-1$ or $s$. 

Let $V$ denote $W$ considered as a vector space of dimension $ms=n$ over $\mathbb{F}_{q}$. Using the trace form from $\mathbb{F}_{q^m}$
to $\mathbb{F}_{q}$, as we did in the proof of Theorem \ref{realizing_dimension_2n}, we may then obtain a subspace, $\M$, say, of $\Symm(V)$ 
of dimension $m(s+1)=2n-r$, in which each non-zero element has rank  $m(s-1)=r$ or $ms=n$.

\end{proof}

It seems likely, especially taking account of the divisibility hypothesis in Theorem \ref{bounds_attained}, that $n-r$ must divide $n$ in
the circumstances of Theorem \ref{symmetric_dimension_bound}, but we have not succeeded in proving this. Theorems 
\ref{dimension_2n} and \ref{symmetric_dimension_bound}
are somewhat similar, but in one case the divisibility condition that $n-r$ must divide $n$ emerges naturally, whereas
in the other this condition is not apparent. We note that, if $n$ above is odd, then $n-r$ must also be odd, and hence $r$ is even; thus
$n$ are $r$ have opposite parity. This difference in parity is not guaranteed if $n$ is even. It is possible that the parity hypothesis
is actually redundant in Theorem \ref{symmetric_dimension_bound}. The next theorem provides a small confirmation that this might
be true.

\begin{theorem} \label{four_dimensional_case}
 Let $V$ be a four-dimensional vector space over $\mathbb{F}_q$, where $q$ is odd.
 Let $\mathcal{M}$ be subspace of $\Symm(V)$ in which
the rank of any non-zero element is $2$ or $4$. Then we have $\dim \M\leq 6$. 

Moreover, suppose that $\dim \M=6$. Then if $u$ is a non-zero element of $V$,  $\M_u$ is a two-dimensional constant rank $2$ subspace of $\M$
and  $\M$ contains exactly $q^4-1$ elements of rank $2$.
\end{theorem}

\begin{proof}

Let $u$ be a non-zero element of $V$. We know from Lemma \ref{estimate_for_dimension_of_kernel} that
\[
 \dim \M\leq \dim \M_u+4
\]
and $\M_u$ is a constant rank 2 subspace of $\Symm(V)$. 

Let $U$ be the one-dimensional subspace spanned by $u$ and let $\overline{V}$ denote the quotient space $V/U$. Each non-zero
element of $\M_u$ induces an element of $\Symm(\overline{V})$ of rank 2. Since we are working over a finite field of odd
characteristic, Corollary 1 and Theorem 7 of \cite{Gow2} imply that $\dim \M_u\leq 2$. It follows that $\dim\M\leq 6$, as claimed.

Let us now suppose that $\dim \M=6$. The argument above shows that $\dim \M_u\leq 2$. On the other hand, Lemma \ref{estimate_for_dimension_of_kernel}
implies that $\dim \M_u\geq 2$ and hence $d(u)=2$ for all non-zero $u$ in $V$. 

Theorem \ref{common_zeros} yields that 
\[
\Sigma_{u\in V}q^{d(u)}=q^6+(q^4-1)q^2=q^4+A_2q^2+A_4=q^4+q^6-1+(q^2-1)A_2.
\]
We readily obtain $A_2=q^4-1$, as required.

\end{proof}

\section{Subspaces of $\Symm(V)$ where three non-zero ranks occur}

\noindent Our intention in this section is to use Theorem \ref{symmetric_dimension_bound} to obtain an upper bound for the dimension
of a subspace of $\Symm(V)$ in which three non-zero ranks occur, one of which is $n$, and where the other two ranks have opposite parities.
We then specialize to the case when these three ranks are $n$, $n-1$ and $n-2$, and obtain fairly detailed information when $n$ is odd.

\begin{theorem} \label{three_ranks}
Let $\M$ be a subspace of $\Symm(V)$ in which the rank of any non-zero element is either $r$, $\ell$ or $n$, 
where $r$ and $\ell$ are positive integers of opposite parity satisfying $r<\ell<n$. Then we have $\dim \M\leq 3n-r-2$,
and $\dim \M_u\leq 2n-2-r$ for all non-zero $u$ in $V$, unless possibly when $q=2$.
\end{theorem}

\begin{proof}
We use the notation of Lemma \ref{estimate_for_dimension_of_kernel}. Let $u$ be a non-zero element of $V$. We have $\dim \M\leq n+\dim \M_u$ 
and each non-zero element
of $\M_u$ has rank $r$ or $\ell$. 

Let $U$ denote the one-dimensional subspace of $V$ spanned by $u$ and  $\overline{V}$ denote the quotient space $V/U$. 
Then each non-zero element of $\M_u$ induces an element
of $\Symm(\overline{V})$ of rank $r$ or $\ell$. Thus, since $\dim \overline{V}=n-1$ and $r$ and $\ell$ are assumed to have opposite parity, 
Theorem \ref{symmetric_dimension_bound} implies that $\dim \M_u\leq 2(n-1)-r$, except possibly if  $q=2$. Thus $\dim \M\leq 3n-r-2$, 
except possibly if $q=2$.
\end{proof}

The bound just obtained for $\dim \M$ is unlikely to be optimal, except in the case when $r=n-2$, a case we shall investigate more
thoroughly in this section. Consider, for example, what happens when $r=n-6$ and $\ell=n-3$. Theorem \ref{three_ranks} yields that
$\dim \M\leq 2n+4$. Now if 3 divides $n$, it is certainly possible to construct by field extension techniques subspaces $\N$ of dimension $2n$
in which only the three non-zero ranks $n-6$, $n-3$ and $n$ occur. We suspect that $2n$ is actually the best bound in this case, not $2n+4$.

\begin{corollary} \label{rank_at_least_n-2}
 Suppose that $n\geq 3$. Let $\M$ be a subspace of $\Symm(V)$ in which each non-zero element has rank at least $n-2$. Then $\dim \M\leq 2n$, 
 except possibly if both $n$ is even and $q=2$. Furthermore, except possibly if $n$ is even and $q=2$, we have $\dim \M_u=n$ for all non-zero
 elements $u$ in $V$. 
\end{corollary}

\begin{proof}

The dimension bound follows from Theorem \ref{three_ranks} when we take $r=n-2$ and $\ell=n-1$, provided $q\neq 2$. 

Suppose that $\dim \M=2n$. Then we know that $\dim \M_u\geq \dim \M-n=n$ by Lemma \ref{estimate_for_dimension_of_kernel}. However,
Theorem \ref{three_ranks} shows that 
when  $r=n-2$ and $\ell=n-1$, we also have $\dim \M_u\leq n$. Thus $\dim \M_u=n$. 
\end{proof}

The following result has some interest of its own and is also useful to investigate the structure of a subspace of $\Symm(V)$ in which
the rank of a non-zero element is at least $n-2$.

\begin{lemma} \label{one_to_one_correspondence}
 Suppose that $n$ is an even positive integer and let $\M$ be a subspace of $\Symm(V)$ of dimension $n+1$  in which each non-zero element has
 rank at least $n-1$. Then there is a one-to-one correspondence between one-dimensional subspaces of $\M$ spanned by elements of rank $n-1$ and 
 one-dimensional
subspaces of $V$, defined as follows: we associate a one-dimensional subspace of $\M$ spanned by an element of rank $n-1$ with its radical.
\end{lemma}

\begin{proof}
 Theorem \ref{equality_of_dimension} implies that $\dim \M_u=1$ for all non-zero $u$ in $V$. We then argue as we did
 at the conclusion of the proof of Corollary \ref{rank_n-1}. 
\end{proof}

We proceed to obtain more detailed information about the properties of a $2n$-dimensional subspace of $\Symm(V)$ when $n$ is odd and each non-zero
form has rank at least $n-2$.

\begin{theorem} \label{structure_of_2n_dimensional_subspace}
Suppose that $n\geq 3$ is an odd integer. Let $\M$ be a subspace of $\Symm(V)$ of dimension $2n$ in which each non-zero element has rank at least $n-2$.
Then we have
\[
 A_{n-1}=q^{n-1}(q^n-1), \quad A_{n-2}=\frac{(q^n-1)(q^{n-1}-1)}{q^2-1}.
\]
Furthermore, there is a one-to-one correspondence between one-dimensional subspaces of $\M$ spanned by elements of rank $n-2$ and two-dimensional
subspaces of $V$, defined by associating a one-dimensional subspace of $\M$ spanned by a form of rank $n-2$ with its radical.

Each subspace $\M_u$ is $n$-dimensional and if $v$ is linearly independent of $u$ in $V$, $\M_u\cap \M_v$ is one-dimensional and is the 
subspace of $\M$
corresponding to the two-dimensional subspace of $V$ spanned by $u$ and $v$. 
\end{theorem}
 \begin{proof}
  
 Let $u$ be any non-zero element of $V$. Corollary \ref{rank_at_least_n-2} shows
that $\dim \M_u=n$. 
 We may identify $\M_u$ with a subspace of $\Symm(\overline{V})$, where 
$\dim \overline{V}=n-1$, and each non-zero
element of $\M_u$ has rank at least $n-2$. Since $n-1$ is even under our hypothesis, $\M_u$ contains exactly $q^n-q^{n-1}=q^{n-1}(q-1)$ elements
of rank $n-1$, by Theorem \ref{equality_of_dimension}.

Let $v$ be linearly independent of $u$ in $V$. It is clear that $\M_u\cap \M_v$ contains no elements of
rank $n-1$, for any such element would contain both $u$ and $v$ in its radical. It follows
that as there are $(q^n-1)/(q-1)$ one-dimensional subspaces in $V$, 
$\M$ contains exactly
\[
 q^{n-1}(q-1)\times \frac{q^n-1}{q-1}=q^{n-1}(q^n-1)
\]
elements of rank $n-1$, and this is thus the value of $A_{n-1}$. 

When we use the equality $A_{n-2}+A_{n-1}+A_n=q^{2n}-1$, and the fact that $\dim \M_u=n$, Theorem \ref{common_zeros} yields that
\[
 (q^n-1)^2=A_{n-1}(q-1)+A_{n-2}(q^2-1).
\]
As we have already obtained the value of $A_{n-1}$, the stated value for $A_{n-2}$ follows from this equation.

 Corollary \ref{two_dimensional_subspaces} shows
that any given two-dimensional subspace $W$ is the radical of an element of $\M$ of rank $n-2$.
Conversely, each element of $\M$ of rank $n-2$ determines a two-dimensional subspace of $V$, namely its radical.

Now, the number of one-dimensional subspaces of $\M$ spanned by elements of rank $n-2$ is $A_{n-2}/(q-1)$. Since
\[
 \frac{A_{n-2}}{q-1}=\frac{(q^n-1)(q^{n-1}-1)}{(q^2-1)(q-1)}
\]
and this latter number is well known to be the number of two-dimensional subspaces in $V$, we have a one-to-one correspondence between
one-dimensional subspaces of $\M$ spanned by elements of rank $n-2$ and two-dimensional subspaces of $V$.

Continuing with the assumption that $u$ and $v$ are linearly independent vectors, we note that any non-zero element of $\M_u\cap \M_v$ contains
the two-dimensional subspace spanned by $u$ and $v$ in its radical. As we have just proved that there is a unique
one-dimensional subspace spanned by an element of rank $n-2$ whose radical is any given two-dimensional subspace, $\dim \M_u\cap \M_v=1$.
 \end{proof}
 
 The following gives partial information on an analogue of Theorem \ref{structure_of_2n_dimensional_subspace} when $n$ is even.
 
 \begin{theorem} \label{structure_of_2n_dimensional_subspace_n_even}
Suppose that $n\geq 4$ is an even integer. Let $\M$ be a subspace of $\Symm(V)$ of dimension $2n$ 
in which each non-zero element has rank at least $n-2$. Then each two-dimensional subspace of $V$ occurs as the radical of some element
of rank $n-2$ in $\M$ and hence $A_{n-2}\geq (q^{n}-1)(q^{n-1}-1)/(q^2-1)$. We also have 
 $A_{n-1}\leq q^{n-1}(q^n-1)$ provided that $q\neq 2$.
 
 If these inequalities are strict, some linearly independent forms of rank $n-2$ share the same radical.
 
 \end{theorem}

 \begin{proof}
 
 Corollary \ref{two_dimensional_subspaces} shows
that each two-dimensional subspace is the radical of an element of $\M$ of rank $n-2$. It follows 
  that there are at least as many one-dimensional subspaces of $\M$ spanned by elements of rank $n-2$ as there are two-dimensional of subspaces
 of $V$. We deduce that
\[
 \frac{A_{n-2}}{q-1}\geq \frac{(q^n-1)(q^{n-1}-1)}{(q^2-1)(q-1)}
\]
and this gives us the stated inequality for $A_{n-2}$.

Now if we assume that $q>2$, 
Corollary \ref{rank_at_least_n-2} shows
that $\dim \M_u=n$ for all non-zero $u$. Then, as in the proof of Theorem \ref{structure_of_2n_dimensional_subspace}, we have
\[
 (q^n-1)^2=A_{n-1}(q-1)+A_{n-2}(q^2-1).
\]
This equality and the inequality obtained above for $A_{n-2}$ then implies that $A_{n-1}\leq q^{n-1}(q^n-1)$ .

The final part of the theorem is obvious from what has been stated.

\end{proof}
 
 We finish this section by providing an explicit example of a $2n$-dimensional subspace $\M$ of $\Symm(V)$ in which the rank of each non-zero
 element is at least $n-2$. We do not need to confine ourselves to finite fields to perform our constructions and hence we work in the wider context
 of fields possessing a cyclic Galois extension of degree $n$.
 
 Let $K$ be a field and let $L$ be a cyclic Galois extension of $K$ of finite degree $n$. 
Let $G$ denote the Galois group of $L$ over $K$ and let $\sigma$ be a generator of $G$. 
We will consider $L$ to be a vector space of dimension $n$ over $K$. 

Let $\Tr$ denote the trace form $L\to K$. 
Let $\lambda$ and $\mu$ be arbitrary elements of $L$. We define an element $f_{\lambda, \mu}$ of $\Symm(L)$ by setting
\[
 f_{\lambda, \mu}(x,y)=\Tr (\lambda xy+\mu x\sigma(y)+\mu \sigma(x) y)
\]
for all $x$ and $y$ in $L$. 

It is straightforward to check that the set of all such $f_{\lambda, \mu}$ is a $K$-subspace, $\M$, say, of $\Symm(L)$  of dimension $2n$.
The radical of $f_{\lambda, \mu}$ consists of those elements $x$ that satisfy
\[
 \sigma(\mu)\sigma^2(x)+\sigma(\lambda)\sigma(x)+\mu x=0.
\]
Properties of cyclic extensions imply that, provided $\lambda$ and $\mu$ are not both 0,
the subset of solutions of this equation is a $K$-subspace of dimension at most 2. Thus,
$f_{\lambda, \mu}$ has rank at least $n-2$ in the non-zero case, and $\M$ has the desired property. We note from this description of the 
radical, that each two-dimensional $K$-subspace of $L$ arises as the radical of a unique one-dimensional subspace of $\M$, which is in keeping
with Theorem \ref{structure_of_2n_dimensional_subspace} when $n$ is odd and $K$ finite.

Since $\mathbb{F}_q$ has the cyclic extension $\mathbb{F}_{q^n}$ of degree $n$, this construction certainly applies
to all finite  fields.

To construct subspaces of alternating bilinear forms with special rank properties, we continue with the hypothesis that
$L$ is a cyclic Galois extension of $K$, this time of even degree $n=2m$. 
Let $G$ denote the Galois group of $L$ over $K$ and let $\sigma$ be a generator of $G$. Let $L'$ be the $K$-subspace of $L$ consisting
of all elements $x$ that satisfy $\sigma^m(x)=-x$. $L'$ has dimension $m$ over $K$. 

Let $\lambda$ and $\mu$ be arbitrary elements of $L'$ and $L$, respectively. We define an element $f_{\lambda, \mu}$ of $\Alt(L)$ by setting
\[
 f_{\lambda, \mu}(x,y)=\Tr (\lambda \sigma^m(x)y+\mu x\sigma^{m-1}(y)-\mu \sigma^{m-1}(x) y)
\]
for all $x$ and $y$ in $L$. We note that, when $K$ has characteristic 2, the fact that $f_{\lambda, \mu}$ is alternating depends 
on the property of $L'$ that all its elements have trace 0 in these circumstances. 

We may check that the set of all such $f_{\lambda, \mu}$ is a $K$-subspace, $\M$, say, of $\Alt(L)$  of dimension $3m$. 
The radical of $f_{\lambda, \mu}$ consists of those elements $x$ that satisfy
\[
 \lambda\sigma^m(x)-\mu\sigma^{m-1}(x)+\sigma^{m+1}(\mu x)=0.
\]
When we apply $\sigma^{m+1}$ to this equation, we obtain
\[
 -\sigma^{m+1}(\mu)x-\sigma(\lambda)\sigma(x)+\sigma^2(\mu)\sigma^2(x)=0.
\]

Again, properties of cyclic extensions imply that, provided $\lambda$ and $\mu$ are not both 0,
the subset of solutions of this equation is a $K$-subspace of dimension at most 2. Thus,
$f_{\lambda, \mu}$ has rank at least $n-2$ in the non-zero case, and since this form is alternating, it has rank $n-2$ or $n$.
Thus, $\M$ is a subspace of $\Alt(V)$ of dimension $3m=3n/2$ in which each non-zero element has rank $n-2$ or $n$. 

This construction implies that the bound given in Theorem \ref{dimension_3m} is optimal. 

\section{On the existence of constant rank subspaces}

\noindent In the course of this paper, we have examined subspaces $\M$, say, of $\Symm(V)$ of dimension $n+1$, whose non-zero elements have rank at least $n-1$. 
Here, we provide more detail about the structure of such subspaces by bounding the dimension of a constant rank $n-1$ subspace contained in
$\M$. 

\begin{theorem} \label{bound_for_contained_constant_rank_subspace}
 Suppose that $\dim V=n$, where $n$ is even. Let $\M$ be an $(n+1)$-dimensional subspace of $\Symm(V)$ in which the rank of each non-zero 
 element is at least $n-1$.
Then, provided that $q\geq n$, the dimension of a  constant rank $n-1$ subspace of $\M$ is at most $n/2$.
\end{theorem}

\begin{proof}
 We have $\dim \M_u=1$, by 
 Theorem \ref{equality_of_dimension} for any non-zero element $u$ of $V$.
 Let $\N$ be a non-zero constant rank $n-1$ subspace of $\M$ of dimension $d$. Let $g$ be any non-zero element of $\N$. Since we are assuming that
 $q\geq n$, it follows from Theorem 1 of \cite{Gow2} that  $\rad g$ is totally isotropic for all the forms in $\N$. 
 Thus if $v$ spans  $\rad g$, we have
 $f(v,v)=0$ for all $f\in \N$. We say that $v$ is a non-zero \emph{common isotropic point} for the elements of $\N$.
 
Clearly, our theorem is true if $d=1$. Thus, we may assume that $d\geq 2$. Let $h$ be an element of $\N$ linearly independent of $g$. Each of
$\rad g$ and $\rad h$ is one-dimensional, and we claim that the  radicals are different. For, if $\rad g=\rad h$, and 
the vector $u$ spans this one-dimensional subspace, then both $g$ and $h$ are contained in the subspace $\M_u$. 
It follows that $\dim \M_u\geq 2$, and this contradicts our opening
observation. Thus, the radicals are different.

If we now count the non-zero common isotropic points for $\N$, we find that there are at least $q^d-1$ of them, since each one-dimensional
subspace of $\N$ contributes $q-1$ non-zero common isotropic points and we have just shown that different one-dimensional 
subspaces contribute subsets of
$q-1$ common isotropic points that have empty intersection.

Finally, since the elements of $\N$ have odd rank, the total number of non-zero common isotropic points for the elements of $\N$ is
$q^{n-d}-1$. This follows from Theorem 5 of \cite{Du} if $q$ is odd and from the Remark following
Theorem 3 of \cite{Gow3} when $q$ is even. We deduce that $q^d-1\leq q^{n-d}-1$ and hence $2d\leq n$, as required.

\end{proof}

We can improve this bound in the special case that $\dim V=4$.

\begin{theorem} \label{case_n=4}
Suppose that $\dim V=4$. Let $\M$ be a five-dimensional subspace of $\Symm(V)$ in which the rank of each non-zero element is at least $3$.
Then, provided that $q\geq 4$, $\M$ contains no two-dimensional constant rank $3$ subspace.
\end{theorem}

\begin{proof}
 Suppose by way of contradiction that $\M$ contains a two-dimensional constant rank 3 subspace, $\N$, say. Let $f$ and $g$ be linearly independent
elements in $\N$. Then we know from the proof of Theorem \ref{bound_for_contained_constant_rank_subspace} that $f$ and $g$ have different 
radicals.
However, it also follows from Lemma 4 of \cite{Du} that, as $q\geq 4$, $f$ and $g$ have the same radical. 
This contradiction shows that no such constant rank subspace of $\M$ exists.
\end{proof}

Theorem \ref{bound_for_contained_constant_rank_subspace} is unlikely to be optimal,
but there do exist constant rank $n-1$ subspaces of reasonably large dimension
compared with $n$ in the circumstances described. For example, we have shown in the proof of Theorem 5 of
\cite{Gow2} that if $n$ is even and $3$ divides $n+1$, there is 
 an $(n+1)$-dimensional subspace $\M$ of $\Symm(V)$ having the two properties: the rank of each non-zero 
 element of $\M$ is at least $n-1$; $\M$ contains a constant rank $n-1$ subspace of dimension $(n+1)/3$.  
 
 Next, we obtain further information about subspaces $\M$ satisfying the hypotheses of Theorem \ref{structure_of_2n_dimensional_subspace}
 when we are working over sufficiently large finite fields of characteristic 2.

\begin{theorem} \label{characteristic_two_bound_for_contained_constant_rank_subspace}
 Suppose that $\dim V=n$, where $n\geq 3$ is odd, and that $q$ is a power of $2$. 
 Let $\M$ be an $2n$-dimensional subspace of $\Symm(V)$ in which the rank of each non-zero 
 element is at least $n-2$. Then $\M_{\Alt}$ is an $n$-dimensional constant rank $n-1$ subspace.
 
Furthermore, provided that $q\geq n-1$, the dimension of a  constant rank $n-2$ subspace of $\M$ is at most $(2n-3)/3$.
\end{theorem}

\begin{proof}
 
 Theorem \ref{structure_of_2n_dimensional_subspace} shows that $Z(\M)$ consists only of trivial elements. It follows from
 Theorem \ref{size_of_alternating_subspace} that
  $\dim \M_{\Alt}
=n$. In addition, since
 the rank of an alternating bilinear form is even, each non-zero element of $\M_{\Alt}$ has rank $n-1$. Thus $\M_{\Alt}$ is a constant rank
 $n-1$ subspace of dimension $n$. 
 
 Let $\N$ be a constant rank $n-2$ subspace of $\M$ of dimension $d$. We set 
 \[
 V(\N)=\{ v\in V: f(v,v)=0 \mbox{ for all } f\in \N\}.
\]
As proved in \cite{Gow3}, $V(\N)$ is a subspace of dimension $n-\dim \N=n-d$. Furthermore, given the hypothesis that $q\geq n-1$, we showed
in Section 2 of \cite{Gow3} that $\rad f$ is contained in $V(\N)$ for any non-zero element $f$ of $\N$.

Now Theorem \ref{structure_of_2n_dimensional_subspace} shows that if $f$ and $g$ are linearly independent elements of $\N$, $\rad f$ and $\rad g$
are different two-dimensional subspaces of $V$. Consequently, the radicals of the one-dimensional subspaces
of $\N$ constitute $(q^d-1)/(q-1)$ different subspaces. We have also noted that these are all two-dimensional subspaces of the space
$V(\N)$ of dimension $n-d$. We set $m=n-d$. Since the number of two-dimensional subspaces of a space of dimension $m$ is
\[
 \frac{(q^{m}-1)(q^{m-1}-1)}{(q-1)(q^2-1)},
\]
we deduce that
\[
 \frac{q^d-1}{q-1}\leq \frac{(q^{m}-1)(q^{m-1}-1)}{(q-1)(q^2-1)}.
\]
Thus
\[
 (q^2-1)(q^d-1)\leq (q^{m}-1)(q^{m-1}-1).
\]

We clearly must have $m\geq 2$. Suppose first that $m=2$. Then since $V(\N)$ contains $\rad f$ for each non-zero element $f$
of $\N$, we must have $\rad f=V(\N)$. But we also know that linearly independent elements of $\N$ have different radicals. This means that
$\N$ must be one-dimensional, and hence $d=1$. Since $1\leq (2n-3)/3$ when $n\geq 3$, our bound holds in this case.

We can now assume that $m\geq 3$. It is easy to establish the estimate
\[
 \frac{q^{m}-1}{q^{2}-1}\leq 2q^{m-2},
\]
since $q\geq 2$. We then obtain that
\[
 q^d-1\leq 2q^{m-2}(q^{m-1}-1)<2q^{2m-3}.
\]
This clearly implies that $d\leq 2m-3$ if $q\geq 4$. 

If $q=2$, our inequality is
\[
 3(2^d-1)\leq (2^{m}-1)(2^{m-1}-1)<2^{2m-1}.
\]
This yields
\[
 2^{d+1}+2^d-3<2^{2m-1}
\]
and this leads again to the conclusion that $d\leq 2m-3$.

Finally, since $m=n-d$, this inequality for $d$ is $d\leq 2n-2d-3$ and thus $d\leq (2n-3)/3$.

\end{proof}

We can also improve the final dimension estimate for a constant rank $n-2$ subspace of $\M$ when $n=5$, and in this case
we do not need to restrict to characteristic 2.

\begin{theorem} \label{case_n=5}
Suppose that $\dim V=5$. Let $\M$ be a ten-dimensional subspace of $\Symm(V)$ in which the rank of each non-zero element is at least $3$.
Then, provided that $q\geq 4$, $\M$ contains no two-dimensional constant rank $3$ subspace.
\end{theorem}

\begin{proof}
 Let $f$ and $g$ be linearly independent elements of rank 3 in $\M$. 
 Theorem \ref{structure_of_2n_dimensional_subspace} implies
 that $\rad f$ and $\rad g$ have different radicals. On the other hand, if all non-trivial
linear combinations of $f$ and $g$ have rank 3, it also follows from Lemma 4 of \cite{Du} that, as $q\geq 4$, $f$ and $g$ have the same radical. 
Thus no two-dimensional constant rank 3 subspace of $\M$ exists.
\end{proof}


Let $K$ be an arbitrary field and let $L$ be a separable extension field of $K$ of finite degree $n+1$, where $n\geq 2$. Let  $\Tr:L\to K$ 
denote the trace form. Let $\N$ be the $K$-subspace of all elements
 $f_z\in \Symm(L)$ defined by 
\[
 f_z(x,y)=\Tr(z(xy))
\]
for all $x$, $y$ and $z$ in $L$. Each non-zero element of $\N$ has rank $n+1$ and $\dim \N=n+1$. 

Let $W$ be  the $K$-subspace of dimension $n$ in $L$, consisting of all elements of trace zero. Let $\mathcal{M}$ 
be the $K$-subspace of $\Symm(W)$ obtained by restricting each element of $\N$ to $W\times W$. Then, as we saw in the proof
of Theorem \ref{bounds_attained}, each non-zero element of $\M$ has rank at least $n-1$ and $\dim \M=n+1$. 

We wish to show that $\M$ has an interesting property with respect to containing constant rank $n-1$ subspaces, provided that
$|K|$ is sufficiently large and $L$ has no proper subfields containing $K$. This last condition means that $n+1$ is a prime when $K$ is finite.
The following auxiliary field-theoretic lemma is taken from \cite{Matt} and is included to keep this paper reasonably self-contained.

\begin{lemma} \label{no_two_dimensional_subspaces}
Let $K$ be a field and let $L$ be an extension field of $K$ of degree $d>2$. Suppose that there are no proper subfields of $L$ containing $K$.
Let $A$ be any $K$-subspace of $L$ of dimension $d-1$. Then, provided that $|K|\geq d-1$, there does not exist a two-dimensional $K$-subspace $U$
of $L$ such that $U^{\times}\leq A^{-1}$. Here, $U^\times$ denotes the set of non-zero elements of $U$ and $A^{-1}$ is the set of inverse elements
of the non-zero elements of $A$.
\end{lemma}

\begin{proof}
 Suppose by way of contradiction that there is a two-dimensional $K$-subspace $U$
of $L$ such that $U^\times\leq A^{-1}$. Let $z$ be a non-zero element of $U$. Then $z^{-1}U$ is a two-dimensional subspace containing $1$, and
$(z^{-1}U)^\times\leq z^{-1}A^{-1}=(zA)^{-1}$. If we replace $U$ by $z^{-1}U$ and $A$ by $zA$, it suffices to show the non-existence of $U$
when $1$ belongs to $U$.

Let $\alpha$ be an element of $U\setminus K$. Let $k_1$, \dots, $k_{d-1}$ be $d-1$ different elements of $K$, whose existence is guaranteed by
the hypothesis that $|K|\geq d-1$. Then $1$, $k_1+\alpha$, \dots, $k_{d-1}+\alpha$ are $d$ non-zero elements in $U$. Thus, since we are assuming that $U^\times\leq A^{-1}$, 
\[
 1, (k_1+\alpha)^{-1}, \ldots, (k_{d-1}+\alpha)^{-1}
\]
are $d$ elements in $A$, and are hence linearly dependent. Hence, there exist elements $\beta_0$, $\beta_1$, \dots, $\beta_{d-1}$, not all zero, with
\[
 \beta_0+\sum_{i=1}^{d-1} \beta_i (k_i+\alpha)^{-1}=0.
\]

Next, we introduce the polynomial $\Phi$ of degree $d-1$ in $K[x]$ defined by
\[
 \Phi=\prod_{i=1}^{d-1}(x+k_i)
\]
and further define polynomials $\Phi_i$, for $1\leq i\leq d-1$, by
\[
 \Phi_i=\frac{\Phi}{x+k_i}.
\]
We also set $\Phi_i(-k_i)=\lambda_i\neq 0$. We note that $\Phi(\alpha)\neq 0$. 

Let $F\in K[x]$ be defined by
\[
 F=\beta \Phi+\sum_{i=1}^{d-1} \beta_i \Phi_i.
\]
We claim that $F$ is not the zero polynomial. For, since $\Phi_i(-k_j)=0$ if $j\neq i$, we have
\[
 F(-k_j)=\beta_j\lambda_j.
\]
Thus if $F$ is the zero polynomial, $\beta_j=0$ for $1\leq j\leq d-1$, and hence $\beta_0=$ also, since
\[
 \beta_0+\sum_{i=1}^{d-1} \beta_i (k_i+\alpha)^{-1}=0.
\]
This is not the case. Hence $F$ is a non-zero polynomial of degree at most $d-1$, and we can check that $F(\alpha)=0$. 
It follows that the subfield $K(\alpha)$ of $L$ has degree at most $d-1$ over $K$, and since $K(\alpha)\neq K$, this contradicts the fact that
$L$ has no proper subfields containing $K$. Thus there is no two-dimensional subspace $U$ with $U^\times\leq A^{-1}$, as required.
\end{proof}

\begin{theorem} \label{small_constant_rank_subspaces}

Let $K$ be an arbitrary field and let $L$ be a separable extension field of $K$ of finite degree $n+1$, where $n\geq 2$.
Suppose that there are no proper subfields of $L$ containing $K$.
Let  $\Tr:L\to K$ 
denote the trace form. Let $\N$ be the $K$-subspace of all elements
 $f_z\in \Symm(L)$ defined by 
\[
 f_z(x,y)=\Tr(z(xy))
\]
for all $x$, $y$ and $z$ in $L$. 

Let $W$ be  the $K$-subspace of dimension $n$ in $L$, consisting of all elements of trace zero. Let $\mathcal{M}$ 
be the $K$-subspace of $\Symm(W)$ obtained by restricting each element of $\N$ to $W\times W$. Then $\dim \M=n+1$ and
 each non-zero element of $\M$ has rank at least $n-1$. Moreover, provided that $|K|\geq n$, 
 the maximum dimension of a constant rank $n-1$ subspace of $\M$ is one.

\end{theorem}

\begin{proof}
 Let $F_z$ denote the restriction of $f_z$ to $W\times W$. Suppose that $F_z$ has rank $n-1$ and let $y\in W$ span the radical of $F_z$.
 Then we have 
 \[
  \Tr(xyz)=0
 \]
 for all $x$ in $W$. It follows from the non-degeneracy of $\Tr$ and the definition of $W$ that $yz$ is an element of $K$, say $yz=\lambda$,
 where $\lambda$ is a non-zero element of $K$. Thus $z=\lambda y^{-1}$.
 
 Suppose now that $\M$ contains a two-dimensional constant rank $n-1$ subspace. Our analysis above shows that $L$ contains a two-dimensional
 subspace $U$, say, with $U^\times\leq W^{-1}$. However, this possibility is excluded, by Lemma \ref{no_two_dimensional_subspaces}.
 Thus, the maximum dimension of a constant rank $n-1$ subspace of $\M$ is one.

\end{proof}

\begin{corollary} Suppose that $n=r-1$, where $r$ is an odd prime. Let $V$ be a vector space of dimension $n$ over $\mathbb{F}_q$.
 Then there is an $(n+1)$-dimensional subspace $\M$, say, of $\Symm(V)$ in which
 each non-zero element of $\M$ has rank at least $n-1$. Moreover, provided that $|q|\geq r+1$, 
 the maximum dimension of a constant rank
 $n-1$ subspace of $\M$ is one.
\end{corollary}

\begin{proof}
 
 We take $K=\mathbb{F}_q$ and $L=\mathbb{F}_{q^r}$ in Lemma \ref{small_constant_rank_subspaces}. Since $r$ is a prime, there are no
 subfields of $L$ that properly contain $K$. The result follows from the lemma.
\end{proof}

\begin{corollary} \label{algebraic_number_field}
 Let $K$ be an algebraic number field and let $W$ be a vector space of dimension $n\geq 2$ over $K$.  
 Then there is an $n+1$-dimensional subspace $\M$, say, of $\Symm(W)$ in which
 each non-zero element of $\M$ has rank at least $n-1$. Moreover, 
 the maximum dimension of a constant rank
 $n-1$ subspace of $\M$ is one.
 
\end{corollary}

\begin{proof}
 Hilbert's irreducibility theorem implies that there is an irreducible polynomial $P$, say, in $K[x]$ whose Galois group over $K$
 is isomorphic to the symmetric group $S_{n+1}$ of degree $n+1$. Let $M$ be a splitting field for $P$ over $K$ and let $L$ be the fixed
 field of a subgroup of $S_{n+1}$ isomorphic to $S_n$. 
 
 Since $S_n$ is a maximal subgroup of $S_{n+1}$, it follows from the Galois correspondence that $L$ is a minimal subfield
  of $M$ containing $K$. Thus the result follows from Theorem \ref{small_constant_rank_subspaces}.
\end{proof}

\end{document}